\documentclass[a4paper,11pt]{amsart}
\usepackage{hyperref,latexsym}
\usepackage{enumerate}

\theoremstyle{plain}
\newtheorem{theorem}{Theorem}[section]
\newtheorem{lemma}[theorem]{Lemma}

\theoremstyle{definition}
\newtheorem{definition}[theorem]{Definition}

\theoremstyle{remark}
\newtheorem{remark}{Remark}

\newcommand{\abs}[1]{\left|#1\right|}

\begin{document}

\title[Smooth solutions]
      {Smooth solutions for a ${p}$-system \\ of mixed type}

\date{21 March 2012}
\author{Misha Bialy}
\address{School of Mathematical Sciences, Raymond and Beverly Sackler Faculty of Exact Sciences, Tel Aviv University,
Israel} \email{bialy@post.tau.ac.il}
\thanks{Partially supported by ISF grant 128/10}


\begin{abstract}
In this note we analyze smooth solutions of a $p$-system of the
\textit{mixed} type. Motivating example for this is a 2-components
reduction of the Benney moments chain which appears to be connected
to theory of integrable systems. We don't assume a-priory that the
solutions in question are in the Hyperbolic region. Our main result
states that the only smooth solutions of the system which are
periodic in $x$ are necessarily constants. As for initial value
problem we prove that if the initial data is strictly hyperbolic and
periodic in $x$ then the solution can not extend to $[t_0;+\infty)$
and shocks are necessarily created.

\end{abstract}

\maketitle

\section{Motivation and the result}
\label{sec:intro} Consider the \textit{p}-system
\begin{equation}\label{1}
u_t=-v_x ,\  v_t=(p(u))_x
\end{equation}
We refer to \cite {Serre} and \cite {Slemrod} for general theory of
these systems.

The general case considered in this note is somewhat different. We
assume that $p$ is a $C^2$-function, which is quadratic like, that
is everywhere strictly convex
$$p''(u)>0,$$  and has a minimum. We shall assume that the minimum value is zero at $u=0.$
Notice that this is not a restriction and can be achieved by
changing $u$ and $p(u)$ by constants. Thus the system is of a
\textit{mixed type} depending on the sign of $u$ (see below), so
that one needs to analyze the solutions in Hyperbolic and Elliptic
regions.

Main example for me which motivates in fact the result of this paper
is $p(u)=u^2/2$. In this case the system appears to be a reduction
of the famous Benney chain and turns out to be related to a problem
of Polynomial integrals for Hamiltonian systems with potential. This
relation was first observed in \cite {Kozlov} and smooth periodic
solutions were studied in \cite{B0}, \cite{B1}. More general Rich
quasi-linear systems of order $\geq 2$ were recently discovered in
connection with other integrable Hamiltonian systems \cite{BM1},
\cite{BM2}. The main question arising in these applications is if
there exist global periodic smooth solutions of the corresponding
quasi-linear systems. It is an open problem how to apply the ideas
of this note to study smooth periodic solutions for these systems of
order $\geq 2$.

In this note we prove non-existence of classical solutions $(u(t,x),
v(t,x))$ of the system, which are $C^2$-smooth and periodic in $x$.
\begin{remark}
Notice that for solutions $u,v$ of the system it follows that  $u$
must satisfy the equation $u_{tt}+(p(u))_{xx}=0$. However the
requirement of periodicity in $x$ of both $u$ and $v$ is a stronger
assumption then just to ask  $u$ to be periodic. This can be seen on
the following example. Take $p(u)=u^2/2$ and notice that $u=t$ is a
periodic in $x$ solution of $u_{tt}+(uu_x)_x=0$ however the
corresponding $v$ for the system is $v=-x$ which is not periodic in
$x$.
\end{remark}

The type of the system is governed by the sign of the function
$p'(u)$. It is hyperbolic for $p'(u)<0$, that is for $u<0$, and
elliptic for $p'(u)>0$, that is for $u>0$. I shall denote by
$$U_+=\{(t,x):u(t,x)>0\}, U_-=\{(t,x):u(t,x)<0\},$$
the elliptic and the hyperbolic domains on the cylinder
respectively.

Our main results are given in the following two theorems below:
\begin{theorem}
Let the pair $(u,v)$ be a solution of the $p$-system of class $C^2$
on the semi-infinite cylinder $t \geq t_0$. If the initial data for
$t=t_0$ is hyperbolic, i.e. $u(t_0,x)< 0$ for all $x$ then it is
hyperbolic everywhere and moreover the solution is constant,
$$(u,v)=const.$$
\end{theorem}
Let me emphasize that this is important for the proof of this
theorem that initially the solution has hyperbolic type and it is
possible, though I don't know any single example, that there exist
smooth solutions of the \textit{mixed} type on the semi-infinite
cylinder. On the other hand, we prove that if one considers the
whole infinite cylinder one gets an additional rigidity and no
smooth solutions of the system other than constants can exist:
\begin{theorem}
Let $(u,v)$ be a $C^2$-smooth solution of the $p$-system on the
whole cylinder. Then both $u$ and $v$ are constant functions.

\end{theorem}

One should mention that the classical Hodograph method  \cite {CF}
can be used to linearize the system of equations. Such a
linearization leads to a second order equation of mixed type (see
recent monograph \cite {T} for extensive survey). This method works
well in neighborhood of a regular points of the mapping
$(x,y)\rightarrow (u,v)$. However in order to make this method
global one needs to analyze different domains where this change of
variables is a diffeomorphism taking into account the singularities
of the mapping $(x,y)\rightarrow (u,v)$ which might be vary
complicated (folds , cusps as well as singularities of rank 2).

Therefore our approach is different and is based on the method of
characteristics by P. Lax \cite{L} in the hyperbolic region together
with certain convexity argument in the elliptic region. Special
attention is payed to the interface between them.

In the sequel we treat the Hyperbolic domain first and then the
Elliptic one. We use the following main ingredients: Near the
boundary of the Hyperbolic region the genuine non-linearity
increases and becomes "infinite"; Analysis of Hyperbolic region is
in fact non-local and uses in a strong way periodicity assumption.
For the Elliptic region we replace maximum principle which works
well for bounded domains by a simple convexity argument.

\section*{Acknowledgements}
It is a pleasure to thank Marshall Slemrod and Steve Schochet for
valuable remarks and stimulating discussions.

\section{Hyperbolic region}
In the Hyperbolic region $u$ and $p'(u)$ are negative and the
boundary of the Hyperbolic region lies in $U_0=\{u=0\}$. We shall
denote by $\lambda_{1,2}=\pm \sqrt {-p'(u)}$ (as usual subindex 1
corresponds to the upper sign here and in the sequel) the
eigenvalues of the matrix
\begin{equation}\label{M}
 A=  \left(
  \begin{array}{cc}
   0 & 1 \\
   -p'(u) & 0
  \end{array}\right),
\end{equation}
with the Riemann invariants
$$r_{1,2}=v\mp \int_u^0 \sqrt {-p'(u)}du$$ and so $u,v$ can be recovered
from the Riemann invariants by the formulas:
\begin{equation}\label{q}
 v=\frac
{r_1+r_2}{2};\ u=q^{-1}\left(\frac{r_2-r_1}{2}\right),
\end{equation}
where by definition
$$q(u):=\int_u^0 \sqrt {-p'(s)}ds,
$$
 is a positive monotone decreasing function for $u<0$
with $$q(0)=0,\  q'(u)=-\sqrt{-p'(u)},\
q''(u)=\frac{p''(u)}{2\sqrt{-p'(u)}}.$$ It is crucial fact that both
eigenvalues are genuinely non-linear by the formulas:
$$(\lambda_1)_{r_1}=(\lambda_2)_{r_2}=\frac{p''(u)}{4p'(u)}.$$
Notice that near the boundary $\partial U_-$ the non-linearity
becomes infinite. Moreover, verifying literarily the Lax method
\cite{L} for our $p$-system one arrives to the following Ricatti
equations along characteristics of the first and the second
eigenvalues:
\begin{equation}\label{R}
L_{v_1}(\beta_1)+k\beta_1^2=0,\ L_{v_2}(\beta_2)+k\beta_2^2=0
\end{equation}
where$$\beta_1:=(r_1)_x (-p'(u))^{\frac{1}{4}};\ \beta_2:=(r_2)_x
(-p'(u))^{\frac{1}{4}} ;\quad k:=-\frac
{p''(u)}{4(-p'(u))^{\frac{5}{4}}} ,$$ and $L_{v_1},L_{v_2}$ stands
for derivatives along the first and the second characteristic fields
respectively. The following two lemmas will be very useful for the
proofs.
\begin{lemma}

(1)\ If a characteristic curve of the first or of the second
eigenvalue starting from the initial time $t_0$ reaches the boundary
$\partial U_-$ in a finite time $t_+>t_0$(respectively $t_-<t_0$),
then the derivative of the corresponding Riemann invariant $r_x$
must be non-positive (resp. non-negative) along this characteristic.

(2)\ If a characteristic curve of the first or of the second
eigenvalue extends to a semi-infinite interval $(t_0,+\infty)$
(resp.$(-\infty,t_0)$), then  for the corresponding Riemann
invariant either  $(r)_x \leq 0$ (resp.$r_x\geq 0$) or otherwise
$-u$ tends to $ +\infty$ along this characteristic curve when $t
\rightarrow +\infty$ (resp.$t \rightarrow -\infty$).

\end{lemma}
\begin{proof}
The proof easily follows from the exact formulas for
the solutions of (\ref{R}):
$$\beta(t)=\frac{\beta(t_0)}{1+\beta(t_0)\int_{t_0}^t k(s)ds}.$$
Indeed in order to prove (1) suppose that characteristic extends to
the maximal interval $(t_0;t_+).$ Recall that the characteristics
are solutions of the ODE's $$\dot{x}=\pm\sqrt{-p'(u)},$$ so when
characteristic curve approaches the boundary of Hyperbolic domain,
so that $u$ tends to zero, then the characteristic curve must
converge to a limit point say $(t_+,x_+)$ on the boundary. Moreover,
it follows then that the integral
$$\int_{t_0}^{t_+}
k(s)ds=-\int_{t_0}^{t_+}\frac{p''(u)}{4(-p'(u))^{\frac{5}{4}}}ds$$
diverges to $-\infty$. Indeed, for $t\rightarrow t_+$ the function
$u(t,x(t))\rightarrow 0$  and can be estimated from above by
$$\abs{u(t, x(t)}\leq C_1 \abs
{t-t_+} ,$$ also for $u$ close to zero one can estimate:
$$\abs{p'(u)}=\abs{\int_u^0 p''(u)du}\leq
C_2\abs{u}, \ where\ C_2=\max_{u\in[-1,0]}p''(u).$$

So the nominator in of the integrand is bounded away from zero and
the denominator is less or equal then $C_1C_2 \abs { t-t_+
}^{\frac{5}{4}}$, thus the integral diverges. This proves the first
part of the lemma.

The second part is analogous. Indeed for an infinite characteristic
there are two possibilities. The first is when the integral diverges
to $-\infty$, in this case $r_x\leq 0$ exactly as in the previous
case. In the second case the integral converges, then in particular
the integrand must tend to zero along characteristic. Then since $p$
is strictly convex it follows that $p'(u)$ tends to $-\infty$ and so
$u$ also. Lemma is proved
\end{proof}

This lemma enables to divide between two types of characteristics
which start at $t_0$ in a positive or negative direction of time as
follows.

\begin{definition}

Let $\gamma$ be a characteristic curve defined on a \textit{maximal}
interval $[t_0,t_+)$ (or respectively $(t_-,t_0]$). We shall say
that $\gamma$ is of type $B_+$ (res. $B_-$) if $t_+=+\infty$ (resp.
$t_-=-\infty$) and $-u\rightarrow +\infty$ when $t\rightarrow
+\infty$ (resp. $t\rightarrow -\infty$).

We shall say that $\gamma$ is of type $A_+$ (resp. $A_-$) in the
opposite case. That is if either $t_+$ (resp. $t_-$) is finite, or
$t_+=+\infty$ (resp. $t_-=-\infty$) and $-u$ does not tend to
$+\infty$ when $t\rightarrow +\infty$ (resp. $t\rightarrow
-\infty$).
\end{definition}

\begin{lemma}
There cannot exist two semi-infinite characteristics in the same
direction
$$\gamma_1=(t,x_1(t)),\ \gamma_2=(t, x_2(t))$$ of the first and the
second eigenvalue belonging both to the same class $B_+$ or $B_-$.

\end{lemma}
\begin{proof}

Assume on the contrary that there exist such $\gamma_1=(t,x_1(t)),\
\gamma_2=(t, x_2(t))$ belonging to the same class, say $B_+$, so
that $-u|_{\gamma_{1,2}}\rightarrow +\infty$ when $t\rightarrow
+\infty$.

 Then
by periodicity we can shift the characteristics to get
$$\gamma_1^{(k)}=(t,x_1(t)+k),\  \gamma_2^{(l)}=(t,x_2(t)+l)$$ for all
$k,l \in \mathbf{Z}$ are characteristics of class $B_+$ also. Since
the functions $x_1,x_2$ are solutions of the ODEs $$\dot{x}=\pm
\sqrt{-p'(u)}$$ respectively, it follows that $x_1$ (respectively
$x_2$) are strictly monotone increasing (resp.decreasing) function
with the derivative bounded away from zero. Therefore for
sufficiently large $k$ the characteristics $\gamma_1$ and
$\gamma_2^{(k)}$  must intersect in a unique point, call it $P_k$.
Denote by $t_k$ the $t$-coordinates of $P_k$. One can see that $t_k$
is monotone increasing and must tend to $+\infty$. Indeed in the
opposite case there exist limits $t_k\nearrow t_*$ and
$P_k\rightarrow P_*$ so that by intermediate value theorem $p'(u)$
is unbounded on the compact segment $[t_0;t_*]$, contradiction.

Therefore we have,
$$-u(t_k,x(t_k))\rightarrow+\infty ,\ k\rightarrow+\infty,$$
and then by formula (\ref{q}) also
$$ r_2(P_k)-r_1(P_k)\rightarrow+\infty ,\ k\rightarrow+\infty.$$
 But this is not possible since by periodicity in $x$ of $(u,v)$ one
 has that
 $r_1(t_0,x)$ and $r_2(t_0,x)$ are bounded, and so by conservation of $r_1,r_2$
 along characteristics $
 r_2(P_k)-r_1(P_k)$ must be bounded also. This contradiction proves the lemma.
\end{proof}

Let me prove now Theorem 1.1.
\begin{proof}
Introduce $$t'= \sup \{t: [t_0,t]\times \mathbf{S}^1 \subseteq
U_-\},
$$ this means that $t'$ is the first moment where non-Hyperbolic
type appears. In other words $u$ must vanish at some point on
$\{t'\} \times \mathbf{S}^1.$  Write $U'_-=[t_0,t')\times
\mathbf{S}^1$. We prove that $t'$ equals in fact to $+\infty$.
Indeed,
 it follows from the second lemma that all characteristics of at least one of the
 eigenvalues are of class $A_+$. Without loss of generality let it be the family of the second
 eigenvalue with this property. Then it follows from the lemma 2.1 that
 $$(r_2)_x(t_0,x)\leq 0,
$$
for every $x$. But by periodicity this is possible only when
$r_2(t_0,x)$ is in fact constant for the initial moment and so also
everywhere on the whole $U'_-$. This means that within the domain
$U'_-$ only $r_1$ can vary. But then $u$ is a function of $r_1$ only
and therefore has constant values along characteristics of the first
eigenvalue. By the construction there exists a point on $\{t'\}
\times \mathbf{S}^1$ where $u$ vanishes. Since there exists a
characteristic of the first family terminating at this point, then
$u$ vanishes also on the whole characteristic. But this is a
contradiction, since $u$ must be strictly negative on $U_-$. This
implies that the hyperbolic domain $U_-$ coincides with the whole
semi-infinite cylinder $t\geq t_0$.

Furthermore since we  know that $r_2$ is a constant on the whole
half cylinder then $u$ depends only on $r_1$ and has constant values
along characteristics of the first family (in particular $-u$ does
not tend to infinity) so by Lemma 2.1 this implies
$$(r_1)_x\leq 0.$$ Using periodicity again we conclude that $r_1$ is constant
also everywhere on the half-cylinder. Thus $(u,v)$ is a constant
solution.
 We are done.
\end{proof}

The idea of the proof of the second theorem is somewhat similar. The
difference is that one should take into account both directions of
characteristics of the Hyperbolic domain. Also in this case elliptic
domains cannot be excluded as before, one needs an additional tool.
In the next section we treat the Elliptic region. These two steps
provide the proof of theorem 1.2. The first step is the following:

\begin{theorem}
Let $(u,v)$ be a $C^2$-solution of the system on the infinite
cylinder. Then either $U_-$ coincides with the whole cylinder and
$u,v$ are constants everywhere, or $U_-$ is empty, i.e $u \geq 0$
everywhere.
\end{theorem}
\begin{proof}To give a proof assume that $U_-$ is not the whole cylinder (otherwise Theorem 1.1 yields the result).
We need to show that $u\geq 0$, everywhere on the cylinder. Fix a
connected component, denote it $U'_-$, and take any initial moment
$t_0$ with the property that the intersection of $\{t=t_0\}$ with
component $U'_-$ is not empty. It consists of the disjoint union of
open intervals- intervals of Hyperbolicity (the case when it is the
whole circle is covered already by theorem 1.1).

Consider the following 2 complementary cases:

Case1. \textit{For any initial moment $t_0$ with the property
$\{t=t_0\}\cap U_-\neq \emptyset$, for at least one of the
eigenvalues (say for the second one) all characteristics started
from $t_0$ in positive and negative direction belong to the classes
$A_+,A_-$. }

In this case by the first lemma $r_2(t_0,x)$ is constant. Since
$t_0$ is arbitrary this implies that $r_2$ is constant on every
connected component of $U_-$. This implies that only $r_1$ varies on
every component and so $u,v, \lambda_1,\lambda_2$ are functions of
$r_1$ only. This means that $u$ keeps constant values along
characteristics of the first eigenvalue. Therefore every such
characteristic can be extended infinitely in both directions because
if it reaches the boundary of the Hyperbolic region $u$ must have
value zero on the whole characteristic, which contradicts
Hyperbolicity. Moreover, since $\lambda_1$ is a function of $r_1$
only so that it is constant along $\lambda_1$-characteristics, then
these characteristics are parallel horizontal straight lines,
\{$x=const$\}. But then $\lambda_1$ must be zero and so again $u=0$,
contradiction.

Case2. \textit{There exists $t_0$ such that for each eigenvalue
$\lambda_1, \lambda_2$ there exists a characteristic of class $B$
started at $t_0$ in some direction.}

It follows from the second lemma that the directions of these
characteristics  must be opposite. So assume without loss of
generality that $\gamma_1,\ \gamma_2$ are $\lambda_1,\
\lambda_2$-characteristics in the classes $B_+,\ B_-$ respectively.
Then it follows from the lemmas that the characteristics $\gamma_1,\
\gamma_2$ being extended beyond $t_0$ in the negative and positive
direction respectively belong to classes $A_-, A_+$ respectively.
And thus by the first lemma
$$(r_1)_x(t_0,x)\geq 0,\ (r_2)_x(t_0,x)\leq 0,$$ for all $x$ in the
intervals of Hyperbolicity. So in this case $(r_1-r_2)(t_0,x)$ is a
monotone function in $x$. Then also $u(t_0,x)$ is monotone  by the
formula (\ref{q}) and since $u$ vanishes at the ends of the
intervals of Hyperbolicity, $u$ must vanish identically. This
contradiction completes the proof of the theorem.

\end{proof}

\section{Elliptic case}
\begin{theorem}
Suppose $(u,v)$ is a $C^2$-solution of the system on the whole
cylinder satisfying $u\geq 0$ everywhere. Then $u,v$ are constants.
\end{theorem}
If the elliptic domain is bounded one could use the maximum
principle for the proof. But in general the following tool can be
used:
\begin{proof}
Take any function $f:\mathbf{R}_+\rightarrow\mathbf{R}_+$ satisfying
$$f(0)=0, \ f(u)>0 \ for \ all\ u>0\  and \ f''(u)<0 \ for \ all\  u \geq 0.$$
Introduce $$E(t)=\int_{S^1}f(u(t,x))dx,$$ which by the construction
is a positive function of $t$ unless $u$ is identically zero.
Compute the second derivative of $E$ using the $p$-system and
integration by parts. We have
$$\ddot{E}=\int_{S_1}f''(u)\left((v_x)^2+p'(u)(u_x)^2\right)dx.$$
Since $u$ is non-negative then $p'(u)$ is non-negative also, since
$p$ is quadratic like by the assumption. So we get that $E$ is a
positive concave function and thus must be constant. Then obviously
$u,v$ are constants everywhere.
\end{proof}

\end{document}